\documentclass[11pt]{article}%
\usepackage{mathrsfs}
\usepackage{bbm}
\usepackage{amsfonts}
\usepackage{amsmath,amssymb}
\usepackage{amsmath}
\usepackage{amssymb}
\usepackage{graphicx}%
\usepackage{shorttoc}
\providecommand{\U}[1]{\protect \rule{.1in}{.1in}}
\newtheorem{theorem}{Theorem}[section]
\newtheorem{corollary}[theorem]{Corollary}
\newtheorem{definition}[theorem]{Definition}

\newtheorem{lemma}[theorem]{Lemma}
\newtheorem{proposition}[theorem]{Proposition}
\newtheorem{remark}[theorem]{Remark}

\newenvironment{proof}[1][Proof]{\noindent \textbf{#1.} }{\  $\Box$}
\numberwithin{equation}{section}

\begin{document}

\title{\textbf{Properties of $G$-martingales with finite variation and  the application to  $G$-Sobolev spaces }}
\author{ Yongsheng Song\thanks{%
Academy of Mathematics and Systems Science, CAS, Beijing, China,
yssong@amss.ac.cn. Research supported  by NCMIS;
Key Project of NSF (No. 11231005); Key Lab of Random Complex
Structures and Data Science, CAS (No. 2008DP173182).}  }

\date{\today}
\maketitle

\begin{abstract} As is known, a process of form $\int_0^t\eta_sd\langle B\rangle_s-\int_0^t2G(\eta_s)ds$, $\eta\in M^1_G(0,T)$, is a non-increasing $G$-martingale. In this paper, we shall show that a non-increasing $G$-martingale could not be form of $\int_0^t\eta_sds$ or $\int_0^t\gamma_sd\langle B\rangle_s$, $\eta, \gamma \in M^1_G(0,T)$, which implies that the decomposition for generalized $G$-It\^o processes is unique: For $\zeta\in H^1_G(0,T)$, $\eta\in M^1_G(0,T)$ and non-increasing $G$-martingales $K, L$, if \[\int_0^t\zeta_s dB_s+\int_0^t\eta_sds+K_t=L_t,\ t\in[0,T],\] then we have $\eta\equiv0$, $\zeta\equiv0$ and $K_t=L_t$.  As an application, we give a characterization to the $G$-Sobolev spaces introduced in Peng and Song (2015).

\end{abstract}


\textbf{MSC-classification}: 60G44, 60G45, 60G48

\section{Introduction}

The notion of $G$-expectation is a type of nonlinear expectation proposed by Peng \cite{Peng-G, Peng-book}. It can be regarded as a
nonlinear generalization of Wiener probability space $(\Omega ,\mathcal{F}, P)$ where $\Omega =C_0([0,\infty) ,\mathbb{R}^{d})$ equipped with the uniform norm, $\mathcal{F}=\mathcal{B}%
(\Omega )$ and $P$ is a Wiener probability measure defined on $(\Omega ,%
\mathcal{F})$. Recall that the Wiener measure is defined such that the
canonical process $B_{t}(\omega ):=\omega _{t}$, $t\geq 0$ is a continuous
process with stationary and independent increments, namely $(B_{t})_{t\geq 0}$
is a Brownian motion. $G$-expectation $\mathbb{E}$ is a sublinear
expectation on the same canonical space $\Omega $, such that the same
canonical process $B$ is a $G$-Brownian motion, i.e., it is a continuous
process with stationary and independent increments. A crucial difference is that the quadratic variance process $\langle B\rangle$ of the $G$-Brownian motion $B$ is no longer a deterministic function of the time variable $t$. It is a process with stationary and independent increments. For the one-dimensional case, its increments are bounded by $\overline{\sigma}^2:=\mathbb{E}[B_1^2]\geq-\mathbb{E}[-B_1^2]=:\underline{\sigma}^2$,
\begin {eqnarray}\underline{\sigma}^2(t-s)\leq \langle B\rangle_t-\langle B\rangle_s\leq\overline{\sigma}^2(t-s), \ \textmd{for} \ s<t.
\end {eqnarray}

Similar to the classical Brownian motion, the $G$-Brownian motion corresponds to a (fully nonlinear) PDE: For a function $\varphi\in C_{b,Lip}(\mathbb{R})$, the collection of bounded Lipstchiz functions on $\mathbb{R}$, the function $u(t,x):=\mathbb{E}[\varphi(x+B_t)]$ is the (viscosity) solution to the following $G$-heat equation
\begin {eqnarray*}
\partial_t u-G(\partial^2_x u)&=&0, \ (t,x)\in (0,\infty)\times \mathbb{R},\\
                        u(0,x)&=& \varphi (x),
\end {eqnarray*} where $G(a)=\frac{1}{2}(\overline{\sigma}^2 a^+-\underline{\sigma}^2 a^-)$, $a\in \mathbb{R}$.  Moreover, for fixed $T>0$, the process $u(T-t,B_t)$, $t\in[0,T]$ is a martingale under $G$-expectation. By It\^o's formula, one has
\begin {eqnarray*}u(T-t,B_t)=& &\mathbb{E}[\varphi(B_T)]+\int_0^t\partial_xu(T-s,B_s)dB_s\\
&+&\frac{1}{2}\int_0^t\partial^2_xu(T-s,B_s)d\langle B\rangle_s-\int_0^tG(\partial^2_xu)(T-s,B_s)ds.
\end {eqnarray*}
The process $M_t:=\int_0^t\partial_xu(T-s,B_s)dB_s$ is a symmetric $G$-martingale (i.e., $M$ and $-M$ are both $G$-martingales), which shares the same properties with classical martingales in the probability space.
The process $K_t:=\frac{1}{2}\int_0^t\partial^2_xu(T-s,B_s)d\langle B\rangle_s-\int_0^tG(\partial^2_xu)(T-s,B_s)ds$ is a non-increasing $G$-martingale. For the linear case ($\underline{\sigma}=\overline{\sigma}$), this term disappears. However, when $\underline{\sigma}<\overline{\sigma}$, $G$-martingales with finite variation are a class of nontrivial processes, which show the variance  uncertainty of $G$-expectation.

For $Z\in H^2_G(0,T)$, $\eta\in M^2_G(0,T)$, \cite {P07b} showed  that a process of form
\begin {eqnarray}\label {GMR}
X_t=X_0+\int_0^tZ_sdB_s+\int_0^t\eta_sd\langle B\rangle_s-\int_0^t2G(\eta_s)ds
\end {eqnarray} is a $G$-martingale, and conjectured that for
any $\xi \in L_{G}^{2}(\Omega_{T})$, the martingale $\mathbb{E}_t[\xi]$ has the representation (\ref{GMR}). \cite{P07b} proved this conjecture for  cylinder random variables of form $\xi=\varphi(B_{t_1}, \cdot\cdot\cdot, B_{t_n})$. For the general case, Soner et al (2011) and Song (2011) proved independently the following $G$-martingale decomposition theorem:
\[\mathbb{E}_t[\xi]=\mathbb{E}[\xi]+\int_0^tZ_sdB_s+\int_0^t\eta_sd\langle B\rangle_s+K_t,\] where $K_t$ is a non-increasing $G$-martingale.

In this paper, our interest concentrates on $G$-martingales with finite variation. In the $G$-expectation space, there are three types of processes whose variation is finite.
\begin{description}
\item[(1)] $L_t=\int_0^t\eta_sds, \  \eta\in M_{G}^{p
}(0,T)$;

\item[(2)] $A_t=\int_0^t\zeta_sd\langle B\rangle_s, \ \zeta\in M_{G}^{p
}(0,T)$;

\item[(3)] $G$-martingales with finite variation.
\end{description}
It is a very important problem to distinguish these three types of processes. Song (2012) distinguished  (1) and  (2) completely:
\begin {equation} \label {c1}\int_0^t\eta_sds=\int_0^t\zeta_sd\langle B\rangle_s,\ t\in[0,T]\Longrightarrow\eta\equiv\zeta=0.
\end {equation}
As an immediate corollary of this result, Song (2012) proved the uniqueness of the representation for $G$-martingales with finite variation. Also, Conclusion (\ref{c1}) implies that the decomposition of $G$-It\^o process is unique, which is crucial for Peng and Song (2015) to define the $G$-Sobolev space $W^{1,2;p}_G(0,T)$.

The main job of this paper is to distinguish $G$-martingales with finite variation from the other two types of processes. For a $G$-martingale of the form $K_t(\varsigma)=\int_0^t\varsigma_sd\langle B\rangle_s-\int_0^t2G(\varsigma_s)ds$, if $K_t(\varsigma)=\int_0^t\eta_sds$  (resp. $\int_0^t\zeta_sd\langle B\rangle_s$),  $t\in [0,T]$, then by Conclusion (\ref{c1}), we get $\varsigma\equiv\eta=0$ (resp. $\varsigma\equiv\zeta=0$). So a $G$-martingale $K_t(\varsigma)$ could not be form of (1) or (2). Here we shall prove this conclusion for general $G$-martingales:
\begin {center} A $G$-martingale with finite variation could not be form of $\int_0^t\eta_sds$ or $\int_0^t\zeta_sd\langle B\rangle_s$.
\end {center}
More precisely, let $K$ be a non-increasing $G$-martingale. If \[K_t=\int_0^t\eta_sds \ (\textmd{resp.} \ \int_0^t\zeta_sd\langle B\rangle_s), \ t\in[0,T],\] we conclude that $K\equiv 0$.

Based on this conclusion, we can prove that the decomposition for generalized $G$-It\^o processes is unique: For $\zeta\in H^1_G(0,T)$, $\eta\in M^1_G(0,T)$ and non-increasing $G$-martingales $K, L$, if \[\int_0^t\zeta_s dB_s+\int_0^t\eta_sds+K_t=L_t,\ t\in[0,T],\] then we have $\eta\equiv0$, $\zeta\equiv0$ and $K_t=L_t$. This turns out to be a very strong result.  Many important conclusions in the context of $G$-expectation theory, including Conclusion (\ref{c1}), can be considered as its immediate corollaries (see Remark \ref {app} for details).The main results of this paper are Theorem \ref {Mthm} and Theorem \ref {gG-Ito}.

Peng and Song (2015) introduced the notion of $G$-Sobolev spaces. In the $G$-Sobolev space $W^{\frac{1}{2},1;p}_{\mathcal{A}_G}(0,T)$  the authors defined solutions to the following path dependent PDEs:
\begin{align}
\begin {split}\label {fPPDE-special-intro}
\mathcal{D}_{t}u+ G(\mathcal{D}^{2}_{x} u)+f(t,u,\mathcal{D}_{x} u)  &  =0, \  \ t\in[0,T),\\
u_T  &  =\xi. %
\end {split}
\end{align} This $W^{\frac{1}{2},1;p}_{\mathcal{A}_G}$-solution corresponds to the solution of the backward SDEs driven by $G$-Brownian motion considered in Hu et al (2014).

 In this paper, as an application of the main results, we shall give a characterization of the $G$-Sobolev space $W_{\mathcal{A}_{G}}^{\frac{1}%
{2},1;p}(0,T)$. The main idea is, just like the liner case, to integrate $\mathcal{A}_Gu=\mathcal{D}_{t}u+ G(\mathcal{D}^{2}_{x} u)$ as one operator, which reduces the regularity requirement for the solutions. To well define the derivative $\mathcal{A}_Gu$ for $u\in W_{\mathcal{A}_{G}}^{\frac{1}%
{2},1;p}(0,T)$, the uniqueness of the decomposition for generalized $G$-It\^o processes plays a crucial role.

The rest of the paper is organized as follows.  In Section 2, we present some basic notions and definitions on the  $G$-expectation theory.  We shall prove the main results in Section 3.  As an application of the uniqueness of the decomposition for generalized $G$-It\^o processes, we shall refine the definition of the $G$-Sobolev space $W^{\frac{1}{2},1;p}_{\mathcal{A}_G}$ in Section 4. In Section 5, as an appendix, we present the wellposedness result of $G$-BSDEs obtained in \cite {HJPS}.

\section{Some definitions and notations about $G$-expectation}
We review some basic notions and definitions on the  $G$-expectation theory. The readers may refer to \cite{Peng-G}, \cite{P07b}, \cite{P08a}, \cite{Peng-book} for more details.

Let $\Omega_T=C_{0}([0,T];\mathbb{R}^{d})$ be the space of all $%
\mathbb{R}^{d}$-valued continuous paths $\omega=(\omega(t))_{t\in[0,T]}$ with $\omega(0)=0$ and let $B_{t}(\omega)=\omega(t)$ be the
canonical process.

Let us recall the definitions of $G$-Brownian motion and its corresponding $%
G $-expectation introduced in \cite{P07b}. For simplicity, here we only consider the one-dimensional case.

Set
\begin{equation*}
L_{ip}(\Omega_{T}):=\{
\varphi(\omega(t_{1}),\cdots,\omega(t_{n})):t_{1},\cdots,t_{n}\in
\lbrack0,T],\  \varphi \in C_{b,Lip}(\mathbb{R}^{n}),\ n\in \mathbb{N}%
\},
\end{equation*}
where $C_{b,Lip}(\mathbb{R}^{n})$ is the collection of bounded Lipschitz
functions on $\mathbb{R}^{n}$.

We are given a function $G:\mathbb{R}\mapsto \mathbb{R} $, for $0\leq \underline{\sigma}^{2}\leq \overline{\sigma}%
^{2}$, by \[G(a):=\frac{1}{2}(\overline{\sigma}^{2}a^{+}-\underline{%
\sigma}^{2}a^{-}).\]

\bigskip \ For each $\xi\in L_{ip}(\Omega_{T})$ of the form
\begin{equation*}
\xi(\omega)=\varphi(\omega(t_{1}),\omega(t_{2}),\cdots,\omega(t_{n})),\  \
0=t_{0}<t_{1}<\cdots<t_{n}=T,
\end{equation*}
we define the following conditional  $G$-expectation
\begin{equation*}
\mathbb{E}_{t}[\xi]:=u_{k}(t,\omega(t);\omega(t_{1}),\cdots,\omega
(t_{k-1}))
\end{equation*}
for each $t\in \lbrack t_{k-1},t_{k})$, $k=1,\cdots,n$. Here, for each $%
k=1,\cdots,n$, $u_{k}=u_{k}(t,x;x_{1},\cdots,x_{k-1})$ is a function of $%
(t,x)$ parameterized by $(x_{1},\cdots,x_{k-1})\in \mathbb{R}^{k-1}$, which
is the solution of the following PDE ($G$-heat equation) defined on $%
[t_{k-1},t_{k})\times \mathbb{R}$:
\begin{equation*}
\partial_{t}u_{k}+G(\partial^2_{x}u_{k})=0\
\end{equation*}
with terminal conditions
\begin{equation*}
u_{k}(t_{k},x;x_{1},\cdots,x_{k-1})=u_{k+1}(t_{k},x;x_{1},\cdots x_{k-1},x),
\, \, \hbox{for $k<n$}
\end{equation*}
and $u_{n}(t_{n},x;x_{1},\cdots,x_{n-1})=\varphi (x_{1},\cdots x_{n-1},x)$.

The $G$-expectation of $\xi$ is defined by $\mathbb{E}[\xi]=%
\mathbb{E}_{0}[\xi]$. From this construction we obtain a natural norm $%
\left \Vert \xi \right \Vert _{L_{G}^{p}}:=\mathbb{E}[|\xi|^{p}]^{1/p}$, $p\geq 1$.
The completion of $L_{ip}(\Omega_{T})$ under $\left \Vert \cdot \right \Vert
_{L_{G}^{p}}$ is a Banach space, denoted by $L_{G}^{p}(\Omega_{T})$. The
canonical process $B_{t}(\omega):=\omega(t)$, $t\geq0$, is called a $G$%
-Brownian motion in this sublinear expectation space $(\Omega_T,L_{G}^{1}(%
\Omega_T ),\mathbb{E})$.

\begin {remark} For $\varepsilon\in [0,\frac{\overline{\sigma}^2-\underline{\sigma}^2}{2}]$, set $G_\varepsilon (a)=G(a)-\frac{\varepsilon}{2}|a|$. Sometimes, we denote by $\mathbb{E}_{G_\varepsilon}$ the $G$-expectation corresponds to the function $G_\varepsilon.$
\end {remark}

\begin {definition} A process $\{M_t\}$ with values in
$L^1_G(\Omega_T)$ is called a $G$-martingale if $\mathbb{E}_s(M_t)=M_s$
for any $s\leq t$. If $\{M_t\}$ and  $\{-M_t\}$ are both
$G$-martingales, we call $\{M_t\}$ a symmetric $G$-martingale.
\end {definition}

\begin{theorem}
\label{repr-Gexp} (\cite{DHP11}) There exists a tight subset $\mathcal{P}%
\subset\mathcal{M}_{1}(\Omega_{T})$, the set of probability measures on
$(\Omega_{T},\mathcal{B}(\Omega_{T}))$, such that
\[
\mathbb{E}[\xi]=\sup_{Q\in\mathcal{P}}E_{Q}[\xi] \  \text{for
all}\ \xi\in L_{ip}(\Omega_T).
\]
$\mathcal{P}$ is called a set that represents $\mathbb{E}$.
\end{theorem}

\begin{remark} Let $W_t$ be a one-dimensional standard Brownian motion in the probability space $(\Omega, \mathcal{F}, P)$ and let $\mathbb{F}:=(\mathcal{F}_t)_{t\geq0}$ be the augmented filtration generated by $(W_t)_{t\geq0}$. Denote by $\mathcal{L}_{\mathbb{F}}^G$ the set of $\mathbb{F}$-adapted measurable processes with values in  $[\underline{\sigma},\overline{\sigma}]$. \cite{DHP11} showed that \[\mathcal{P}_G:=\{P_h|P_h:=P\circ(\int_0^\cdot h_sdW_s)^{-1}, h\in \mathcal{L}_{\mathbb{F}}^G\}\] is a set that represents $\mathbb{E}$.
\end{remark}

\begin{definition}
A function $\eta (t,\omega ):[0,T]\times \Omega _{T}\rightarrow \mathbb{R}$
is called a step process if there exists a time partition $%
\{t_{i}\}_{i=0}^{n}$ with $0=t_{0}<t_{1}<\cdot \cdot \cdot <t_{n}=T$, such
that for each $k=0,1,\cdot \cdot \cdot ,n-1$ and $t\in (t_{k},t_{k+1}]$
\begin{equation*}
\eta (t,\omega )=\xi_{t_k}\in L_{ip}(\Omega_{t_k}).
\end{equation*}%
 We denote by $M^{0}(0,T)$ the
collection of all step processes.
\end{definition}

For each $p\geq1$, we denote by $M_{G}^{p}(0,T)$ the completion of the space $M^0(0,T)$ under the norm%
\[
{{\left \Vert \eta \right \Vert _{M_{G}^{p}}:=\left \{  \mathbb{E}%
[{{\int_{0}^{T}|\eta_{t}|^{p}dt]}}\right \}  ^{1/p}}},\  \
\]
and by $H_{G}^{p}(0,T)$ the completion of the space $M^0(0,T)$ under
the norm%
\[
{{\left \Vert \eta \right \Vert _{H_{G}^{p}}:=}}\left[  {{\mathbb{E}%
[\left \{  {{\int_{0}^{T}|\eta_{t}|^{2}dt]}}\right \}  ^{p/2}}}\right]
^{1/p}.\
\]

\begin{theorem}
(\cite{STZ11}, \cite{Song11}) For $\xi\in
L^\beta_G(\Omega_T)$ with some $\beta>1$, $X_t=\mathbb{E}_t(\xi)$, $
t\in[0, T]$ has the following decomposition:
\begin {eqnarray*}
X_t=\mathbb{E}[\xi]+\int_0^tZ_sdB_s+K_t, \ q.s.,
\end {eqnarray*}
 where $\{Z_t\}\in H^1_G(0, T)$  and $\{K_t\}$ is a continuous
 non-increasing $G$-martingale. Furthermore, the above decomposition is unique and
$\{Z_t\}\in H^\alpha_G(0, T)$, $K_T\in L^\alpha_G(\Omega_T)$ for any
$1\leq\alpha<\beta$.
\end{theorem}

\section{Main results}
In the sequel, we shall only consider the one-dimensional $G$-expectation space which is non-degenerate and really nonlinear, i.e., $\overline{\sigma}>\underline{\sigma}>0$.

 Let $W$ be a standard Brownian motion in the probability space $(\Omega, \mathcal{F},  P)$ and assume that $\mathbb{F}=(\mathcal{F}_t)$ is the augmented filtration generated by $W$.

  An $\mathbb{F}$-adapted measurable process $h$ is called an \textit{($m$-steps) self-dependent process} if it has the following form:
\begin {eqnarray}\label {form}h_t=\sum_{i=0}^{m-1}\xi_i1_{]\frac{i}{m}, \frac{i+1}{m}]}(t)
\end {eqnarray}where
$\xi_i=\varphi_i(\int_{\frac{i-1}{m}}^{\frac{i}{m}}h_sdW_s,\cdot\cdot\cdot,\int_0^{\frac{1}{m}}h_sdW_s)$,
$\varphi_i\in C_{b,Lip}(R^i)$.  Clearly,  an $m$-steps self-dependent process $h$ can also be considered as a $2^nm$-steps self-dependent process for any $n\geq 0$.

\begin {lemma} (Lemma 4.2 in \cite {Song12}) \label {density} The
collection of self-dependent processes bounded by two positive constants $c, C$ ($c\leq|h_s|\leq C$) is dense in the collection of $\mathbb{F}$-adapted measurable processes bounded by the same  constants $c, C$ under the norm $$\|h\|_2=[E(\int_0^1|h_s|^2ds)]^{1/2}.$$
\end {lemma}

 Let $B_t(\omega)=\omega_t$ be the canonical process on the space $\Omega_T$. For an  $\mathbb{F}$-adapted measurable process $h$, set $P_h=P\circ(\int_0^\cdot h_sdW_s)^{-1}$, a probability on $(\Omega_T, \mathcal{B}(\Omega_T))$. For a process $\{X_t\}$, we denote  by
$X^m_{[0,T]}$ the vector
$(X_T-X_{\frac{(m-1)T}{m}},\cdot\cdot\cdot, X_{\frac{T}{m}}-X_0)$.

 \begin {lemma} \label {perturbation} Let $h$ be an $m$-steps self-dependent process of form (\ref {form}).  We call  a bounded  $\mathbb{F}$-adapted measurable process $\tilde{h}$ an m-perturbation of $h$ if  the following property holds: \[\int_{\frac{i}{m}}^{\frac{i+1}{m}}|\tilde{h}_s|^2ds=\frac{1}{m}|\tilde{\xi}_i|^2:=\frac{1}{m}|\varphi_i(\int_{\frac{i-1}{m}}^{\frac{i}{m}}\tilde{h}_sdW_s, \cdot\cdot\cdot,  \int_{0}^{\frac{1}{m}}\tilde{h}_sdW_s)|^2, \textit{P-}a.s.\]
 Then for any random variable of the form $X=\psi(B^m_{[0,1]})$ with $\psi$ a bounded Lipschiz continuous function, we have \[E_{P_h}[X]=E_{P_{\tilde{h}}}[X].\]
 \end {lemma}
 \begin {proof} Set $\psi_1(x_{m-1}, \cdot\cdot\cdot,x_1):=E[\psi(\varphi_{m-1}(x_{m-1}, \cdot\cdot\cdot,x_1)(W_1-W_{\frac{m-1}{m}}), x_{m-1}, \cdot\cdot\cdot,x_1)]$. On the one hand, we have
 \begin {eqnarray*}
 E_{P_h}[X]
 &=&E[\psi(\int_{\frac{m-1}{m}}^{1}h_sdW_s, \cdot\cdot\cdot,  \int_{0}^{\frac{1}{m}}h_sdW_s)]\\
 &=&E[E[\psi(\int_{\frac{m-1}{m}}^{1}h_sdW_s, \cdot\cdot\cdot,  \int_{0}^{\frac{1}{m}}h_sdW_s)|\mathcal{F}_{\frac{m-1}{m}}]]\\
 &=&E[\psi_1(\int_{\frac{m-2}{m}}^{\frac{m-1}{m}}h_sdW_s, \cdot\cdot\cdot,  \int_{0}^{\frac{1}{m}}h_sdW_s)].
 \end {eqnarray*}
 On the other hand, letting $P_{\frac{m-1}{m}}^\omega$ be the regular conditional probability of $P(\cdot|\mathcal{F}_{\frac{m-1}{m}})$, $\int_{\frac{m-1}{m}}^{1}\tilde{h}_sdW_s$ is normally distributed under  $P_{\frac{m-1}{m}}^\omega$ with mean 0
 and variance $\frac{1}{m}|\tilde{\xi}_{m-1}|^2(\omega)$ since $\int_{\frac{i}{m}}^{\frac{i+1}{m}}|\tilde{h}_s|^2ds=\frac{1}{m}|\tilde{\xi}_i|^2$. So
 \begin {eqnarray*}
 E_{P_{\tilde{h}}}[X]
 &=&E[\psi(\int_{\frac{m-1}{m}}^{1}\tilde{h}_sdW_s, \cdot\cdot\cdot,  \int_{0}^{\frac{1}{m}}\tilde{h}_sdW_s)]\\
 &=&E[E[\psi(\int_{\frac{m-1}{m}}^{1}\tilde{h}_sdW_s, \cdot\cdot\cdot,  \int_{0}^{\frac{1}{m}}\tilde{h}_sdW_s)|\mathcal{F}_{\frac{m-1}{m}}]]\\
 &=&E[\psi_1(\int_{\frac{m-2}{m}}^{\frac{m-1}{m}}\tilde{h}_sdW_s, \cdot\cdot\cdot,  \int_{0}^{\frac{1}{m}}\tilde{h}_sdW_s)].
 \end {eqnarray*}
 Repeating the above arguments for $m-1$ times, finally we can find a bounded Lipschiz continuous function $\psi_{m-1}$ such that
 \[ E_{P_h}[X]=E_P[\psi_{m-1}(\int_{0}^{\frac{1}{m}}h_sdW_s)], \ E_{P_{\tilde{h}}}[X]=E_P[\psi_{m-1}(\int_{0}^{\frac{1}{m}}\tilde{h}_sdW_s)].\]
 Since $\int_0^t|\tilde{h}_s|^2ds=\frac{1}{m}\tilde{\xi}_0^2=\frac{1}{m}\xi_0^2$, $\int_{0}^{\frac{1}{m}}h_sdW_s$ and $\int_{0}^{\frac{1}{m}}\tilde{h}_sdW_s$ are both normally distributed with mean 0
 and variance $\frac{1}{m}|\xi_{0}|^2$. Hence, we have $ E_{P_h}[X]= E_{P_{\tilde{h}}}[X]$.

 \end {proof}

 \begin {theorem} \label {weak convergence-perturbation} Let $h$ be an $m$-steps self-dependent process. For $n\geq 1$, let $h^n$ be a $2^nm$-perturbation of $h$. Assuming that $(h^n)_{n\geq1}$ are uniformly bounded,  we have \[P_{h^n}\stackrel{w}{\rightarrow }P_h.\]
 \end {theorem}
\begin {proof} For any $k\geq 1$ and any function $\psi\in C_{b,Lip}(R^{2^km})$, by Lemma \ref {perturbation}, we have, for $n\geq k$,
\[E_{P_{h^n}}[\psi(B^{2^km}_{[0,1]})]=E_{P_h}[\psi(B^{2^km}_{[0,1]})].\] In other words, we have
\[\lim_{n\rightarrow\infty}E_{P_{h^n}}[\psi(B^{2^km}_{[0,1]})]=E_{P_h}[\psi(B^{2^km}_{[0,1]})]\] for any $k\geq 1$ and any function $\psi\in C_{b,Lip}(R^{2^km})$.

 Since $(h^n)_{n\geq1}$ are uniformly bounded, we know that $(P_{h^n})_{n\geq1}$ are tight. Combing the above arguments, we conclude that
 \[P_{h^n}\stackrel{w}{\rightarrow }P_h.\]
\end {proof}
\begin {lemma} \label {MP}
Let $K$ be a non-increasing $G$-martingale.   Fix an $\mathbb{F}$-adapted measurable process $h$ with $\underline{\sigma}\leq|h|\leq\overline{\sigma}$. Then for any $s<t$ and any $\varepsilon>0$ there exists an $\mathbb{F}$-adapted measurable process $\tilde{h}$ with $\underline{\sigma}\leq|\tilde{h}|\leq\overline{\sigma}$ and
$\tilde{h}_r1_{[0,s]}(r)=h_r1_{[0,s]}(r)$ such that $E_{P_{\tilde{h}}}[-(K_t-K_s)]<\varepsilon$.
\end {lemma}
\begin {proof} Fix  $s<t$, $\varepsilon>0$ and an $\mathbb{F}$-adapted measurable process $h$ with $\underline{\sigma}\leq|h|\leq\overline{\sigma}$.  By Theorem 5.4 in \cite {Song11b}, for the non-increasing $G$-martingale $K_t$, there exist $\zeta\in M^0(0,T)$ such that
\[\mathbb{E}[\sup_{r\in[0,1]}|K_r-K_r(\zeta)|]<\frac{\varepsilon}{2},\] where $K_r(\zeta)=\int_0^r\zeta_ud\langle B\rangle_u-\int_0^r2G(\zeta_u)du$.
We assume that  $\zeta$ is of the following form:
\[\zeta_u=\sum_{i=0}^{m-1}a_{t_i}1_{]t_i,t_{i+1}]}(u),\] where $a_{t_i}=\phi_i(B_{t_i}-B_{t_{i-1}}, \cdot\cdot\cdot, B_{t_1})$ with $\phi_i\in C_{b,Lip}(R^i)$.
Without loss of generality, we assume $s=t_i$ and $t=t_{i+1}$. Set $\tilde{a}_{t_i}=\phi_i(\int_{t_{i-1}}^{t_i}h_udW_u, \cdot\cdot\cdot, \int_{0}^{t_1}h_udW_u)$ and
\begin {equation}\textmd{sign}_{\underline{\sigma},\overline{\sigma}} (\tilde{a}_{t_i})=
\begin {cases}\overline{\sigma} & \textmd{if $\tilde{a}_{t_i}\geq0$;}\\
\underline{\sigma} & \textmd{if $\tilde{a}_{t_i}<0$.}
\end {cases}
\end {equation}
Let $\tilde{h}_r=h_r$ for $s\in [0,t_i]$ and let  $\tilde{h}_r=\textmd{sign}_{\underline{\sigma},\overline{\sigma}} (\tilde{a}_{t_i})$ for $r\in ]t_i, t_{i+1}]$. Then

\begin {eqnarray*}E_{P_{\tilde{h}}}[K_t(\zeta)-K_s(\zeta)]&=&E_{P_{\tilde{h}}}[a_{t_i}(\langle B\rangle_{t_{i+1}}-\langle B\rangle_{t_i})-2G(a_{t_i})(t_{i+1}-t_i)]\\
&=& E[\tilde{a}_{t_i}\textmd{sign}_{\underline{\sigma},\overline{\sigma}} (\tilde{a}_{t_i})^2(t_{i+1}-t_i)-2G(\tilde{a}_{t_i})(t_{i+1}-t_i)]=0.
\end {eqnarray*}
So
\[E_{P_{\tilde{h}}}[-(K_t-K_s)]\leq 2 \mathbb{E}[\sup_{r\in[0,1]}|K_r-K_r(\zeta)|]<\varepsilon.\]

\end {proof}

\begin {lemma} \label {weak continuity} Let $\mathcal{P}\subset \mathcal{M}_1(\Omega_T)$ be a weakly compact set that represents $\mathbb{E}$:
 \[
\mathbb{E}[\xi]=\sup_{Q\in\mathcal{P}}E_{Q}[\xi] \  \text{for
all}\ \xi\in L_{ip}(\Omega_T).
\]
Then, for $\xi\in L^1_G(\Omega_T)$, $(E_Q[\xi])_{Q\in \mathcal{P}}$ is continuous with respect to the weak convergence topology on $\mathcal{M}_1(\Omega_T)$.
\end {lemma}
\begin {proof} For $\xi\in L_{ip}(\Omega_T)$, $(E_Q[\xi])_{Q\in \mathcal{P}}$ is obviously continuous. By the definition of the space $L^1_G(\Omega_T)$, for  $\xi\in L^1_G(\Omega_T)$, $(E_Q[\xi])_{Q\in \mathcal{P}}$ can be considered as the uniform limit of a sequence of continuous functions $(E_Q[\xi^n])_{Q\in \mathcal{P}}$ with $\xi^n \in L_{ip}(\Omega_T)$ and $\mathbb{E}[|\xi^n-\xi|]\rightarrow0$. So we get the desired result.
\end {proof}

\begin {theorem} \label {Mthm} Let $K_t=\int_0^t\eta_sds$ for some $\eta\in M^1_G(0,T)$. If $K$ is a non-increasing $G$-martingale, we have $K\equiv 0.$
\end {theorem}
\begin {proof} Without loss of generality, we consider the case $T=1$. Assume $\mathbb{E}[-K_1]>0$. Then, by Lemma \ref {weak continuity}, there exists $\varepsilon>0$ such that $\mathbb{E}_{G_\varepsilon}[-K_1]>0$. So, by Theorem \ref {repr-Gexp} and Lemma \ref {density}, we can find  a process $h$ of form (\ref {form}) with $\underline{\sigma}^2+\varepsilon\leq|h_s|^2\leq\overline{\sigma}^2-\varepsilon$ such that $\delta:=E_{P_h}[-K_1]>0$. Set $\alpha=\frac{\varepsilon}{\overline{\sigma}^2-\underline{\sigma}^2}$. For $k\geq1$, set \[\delta_{k, \alpha}(s)=\sum_{i=0}^{k-1}(1_{]\frac{i}{k},\frac{i+\alpha}{k}]}(s)-1_{]\frac{i+\alpha}{k},\frac{i+1}{k}]}(s)).\]
\textbf{Step 1.} For any $n\geq 1$, we can find a $2^nm$-perturbation $h^n$ of $h$  with $\underline{\sigma}\leq|h^n_s|\leq\overline{\sigma}$ such that \[E_{P_{h^n}}[\int_0^1\delta^+_{2^nm,\alpha}(s)\eta_sds]>-\frac{\delta}{2}\alpha.\]
First, let us define $h^n_s$ for $s\in[0, \frac{1}{m}]$.

Set $\xi^n_0=\xi_0$.

By Lemma \ref {MP},  there exists an $\mathbb{F}$-adapted measurable process $h^{n,1,1}$ with $\underline{\sigma}\leq|h^{n,1,1}|\leq\overline{\sigma}$ such that $E_{P_{h^{n,1,1}}}[-(K_{\frac{\alpha}{2^nm}}-K_0)]<\frac{\delta}{2^{n+1}m}\alpha$.

Since $\underline{\sigma}^2+\varepsilon\leq|\xi^n_0|^2\leq\overline{\sigma}^2-\varepsilon$, we have $||h^{n,1,1}_s|^2-|\xi^n_0|^2|\leq \overline{\sigma}^2-\underline{\sigma}^2-\varepsilon$, by which we get $\frac{\alpha}{1-\alpha}||h^{n,1,1}_s|^2-|\xi^n_0|^2|\leq \varepsilon,$ and consequently
\[\frac{2^nm}{1-\alpha}\int_0^{\frac{\alpha}{2^nm}}||h^{n,1,1}_s|^2-|\xi^n_0|^2|ds\leq\varepsilon.\] So, noting $\underline{\sigma}^2+\varepsilon\leq|\xi^n_0|^2\leq\overline{\sigma}^2-\varepsilon$ again, we get
\[|\xi^n_0|^2+\frac{2^nm}{1-\alpha}\int_0^{\frac{\alpha}{2^nm}}|\xi^n_0|^2-|h^{n,1,1}_s|^2ds=
\frac{1}{1-\alpha}(|\xi^n_0|^2-2^nm\int_0^{\frac{\alpha}{2^nm}}|h^{n,1,1}_s|^2ds)\in [\underline{\sigma}^2, \overline{\sigma}^2].\]

Set
\begin {equation}h^n_s=
\begin {cases}h^{n,1,1}_s & \textmd{for $s\in ]0,\frac{\alpha}{2^{n}m}]$;}\\
\sqrt{\frac{1}{1-\alpha}(|\xi^n_0|^2-2^nm\int_0^{\frac{\alpha}{2^nm}}|h^{n,1,1}_s|^2ds)} & \textmd{for $s\in ]\frac{\alpha}{2^{n}m}, \frac{1}{2^nm}]$.}
\end {cases}
\end {equation} It is easy to check that $\int_0^{\frac{1}{2^nm}}|h^n_s|^2 ds=\frac{1}{2^nm} |\xi^n_0|^2$.

Assume that we have defined $h^n_s$ for all $s\in[0,
  \frac{j}{2^nm}]$, $0\leq j\leq 2^n-1$.  Then let us define $h^n_s$ for $s\in]\frac{j}{2^nm}, \frac{j+1}{2^nm}]$.

 By Lemma \ref {MP},  there exists an $\mathbb{F}$-adapted process $h^{n,1,j+1}$ with $\underline{\sigma}\leq|h^{n,1,j+1}|\leq\overline{\sigma}$ and
$h^{n,1,j+1}_r1_{[0,\frac{j}{2^nm}]}(r)=h^n_r1_{[0,\frac{j}{2^nm}]}(r)$ such that $E_{P_{h^{n,1,j+1}}}[-(K_{\frac{j+\alpha}{2^nm}}-K_{\frac{j}{2^nm}})]<\frac{\delta}{2^{n+1}m}\alpha$.

Set
\begin {equation}h^n_s=
\begin {cases}h^{n,1,j+1}_s & \textmd{for $s\in ]\frac{j}{2^{n}m},\frac{j+\alpha}{2^{n}m}]$;}\\
\sqrt{\frac{1}{1-\alpha}(|\xi^n_0|^2-2^nm\int_0^{\frac{\alpha}{2^nm}}|h^{n,1,j+1}_s|^2ds)} & \textmd{for $s\in ]\frac{j+\alpha}{2^{n}m}, \frac{j+1}{2^nm}]$.}
\end {cases}
\end {equation} It is easily seen  that $\int_{\frac{j}{2^{n}m}}^{\frac{j+1}{2^nm}}|h^n_s|^2 ds=\frac{1}{2^nm} |\xi^n_0|^2$ and $|h^n_s|^2\in [\underline{\sigma}^2, \overline{\sigma}^2]$.

Assume that we have defined $h^n_s$ for all $s\in[0,
  \frac{i}{m}]$, $0\leq i \leq m-1$.

Set $\xi^n_i=\varphi_i(\int_{\frac{i-1}{m}}^{\frac{i}{m}}h^n_sdW_s, \cdot\cdot\cdot, \int_{0}^{\frac{1}{m}}h^n_sdW_s)$.

Then we can define the process $h^n_s$ for $s\in ]\frac{i}{m}, \frac{i+1}{m}]$ by repeating the above arguments with $\xi^n_0$ replaced by $\xi^n_i$.

Clealy, the process $h^n_s$ defined in this way is a $2^nm$-perturbation  of $h$. Besides, we have
\[E_{P_{h^n}}[\int_0^1\delta^+_{2^nm,\alpha}(s)\eta_s]=\sum_{j=0}^{2^nm-1}E_{P_{h^n}}[K_{\frac{j+\alpha}{2^nm}}-K_{\frac{j}{2^nm}}]>-\frac{\delta}{2}\alpha.\]

\textbf{Step 2.} $\lim_{n\rightarrow\infty}\mathbb{E}[|\int_0^1\delta^-_{2^nm,\alpha}(s)\eta_sds-(1-\alpha)K_1|]=0$.

For $\zeta\in M^0(0,1)$, the conclusion is obvious.  As a functional of $\zeta\in M^1_G(0,1)$,  \[D_\alpha(\zeta):=\limsup_{n\rightarrow\infty}\mathbb{E}[|\int_0^1\delta^-_{2^nm,\alpha}(s)\zeta_sds-(1-\alpha)\int_0^1\zeta_sds|]\] is continuous: $|D_\alpha(\zeta)-D_\alpha(\varsigma)|\leq \|\zeta-\varsigma\|_{M^1_G}$, for $\zeta, \varsigma\in M^1_G(0,1)$, which implies the desired result.

\textbf{Step 3.} $\lim_{n\rightarrow\infty}E_{P_{h^n}}[\int_0^1\delta^-_{2^nm,\alpha}(s)\eta_sds]=(1-\alpha)E_{P_h}[K_1]=-(1-\alpha)\delta$.

Actually,
\begin {eqnarray*}& & |E_{P_{h^n}}[\int_0^1\delta^-_{2^nm,\alpha}(s)\eta_sds]-(1-\alpha)E_{P_h}[K_1]|\\
&\leq&  |E_{P_{h^n}}[\int_0^1\delta^-_{2^nm,\alpha}(s)\eta_sds]-(1-\alpha)E_{P_{h^n}}[K_1]|+(1-\alpha)|E_{P_{h^n}}[K_1]-E_{P_h}[K_1]|\\
&\leq&  \mathbb{E}[|\int_0^1\delta^-_{2^nm,\alpha}(s)\eta_sds]-(1-\alpha)K_1|]+(1-\alpha)|E_{P_{h^n}}[K_1]-E_{P_h}[K_1]|.
\end {eqnarray*}
By Step 2 and Theorem \ref {weak convergence-perturbation}, we get the desired result.

\textbf{Step 4.} $\lim_{n\rightarrow\infty}\mathbb{E}[\int_0^1(\delta^+_{2^nm,\alpha}(s)-\frac{\alpha}{1-\alpha}\delta^-_{2^nm,\alpha}(s))\eta_sds]=0$.

The proof follows immediately from Step 2. Actually, setting \[ d_\alpha(\zeta):=\limsup_{n\rightarrow\infty}\mathbb{E}[|\int_0^1(\delta^+_{2^nm,\alpha}(s)-\frac{\alpha}{1-\alpha}\delta^-_{2^nm,\alpha}(s))\zeta_sds|], \ \zeta\in M^1_G(0,1),\] it is easily seen that $(1-\alpha)d_\alpha(\zeta)=D_\alpha(\zeta)$.

Combing the above arguments, we get
\begin {eqnarray*}
0&=&\lim_{n\rightarrow\infty}\mathbb{E}[\int_0^1(\delta_{2^nm,\alpha}(s)^+-\frac{\alpha}{1-\alpha}\delta_{2^nm,\alpha}(s)^-)\eta_sds]\\
&\geq&
\limsup_{n\rightarrow\infty}E_{P_{h^n}}[\int_0^1(\delta^+_{2^nm,\alpha}(s)-\frac{\alpha}{1-\alpha}\delta^-_{2^nm,\alpha}(s))\eta_sds]\\
&\geq& \limsup_{n\rightarrow\infty}E_{P_{h^n}}[\int_0^1\delta^+_{2^nm,\alpha}(s)\eta_s]-\frac{\alpha}{1-\alpha}\lim_{n\rightarrow\infty}E_{P_{h^n}}[\int_0^1\delta^-_{2^nm,\alpha}(s)\eta_sds]\\
&\geq&-\frac{\delta}{2}\alpha+\frac{\alpha}{1-\alpha}\times(1-\alpha)\delta=\frac{\delta}{2}\alpha>0,
\end {eqnarray*} which  is a contradiction. The last inequality follows from Step 1 and Step 3.
\end {proof}
\begin {corollary} Let $K_t=\int_0^t\eta_sd\langle B\rangle_s$ for some $\eta\in M^1_G(0,T)$. If $K$ is a non-increasing $G$-martingale, we have $K\equiv 0.$
\end {corollary}
\begin {proof} Set $L_t=\int_0^t\eta_sds$. Assume that $K$ is a non-increasing $G$-martingale. Then we have
\[0\geq\mathbb{E}_s[L_t-L_s]\geq\frac{1}{\underline{\sigma}^2}\mathbb{E}_s[K_t-K_s]=0.\] So $L$ is a non-increasing $G$-martingale. By Theorem \ref {Mthm}, we get $L\equiv0$, and consequently, $K\equiv0.$
\end {proof}

As an application of Theorem \ref {Mthm}, we  shall prove the uniqueness of the decomposition for generalized $G$-It\^o processes.

\begin {definition} A process of the following form  is called a generalized $G$-It\^o process:
\[u=u_0+\int_{0}^{t}\eta_sds+\int_{0}^{t}\zeta_sdB_{s}+K_{t},\] where $\eta\in M^1_G(0,T)$, $\zeta\in H^1_G(0,T)$ and $K$ is a non-increasing $G$-martingale.
\end {definition}
\begin {remark} A $G$-It\^o process \[u=u_0+\int_{0}^{t}\tau_sds+\int_{0}^{t}\zeta_sdB_{s}+\int_0^t\frac{1}{2}\gamma_sd \langle B\rangle_s, \ \tau, \gamma\in M^1_G(0,T), \ \zeta\in H^1_G(0,T),\] can be rewritten as
\[u=u_0+\int_{0}^{t}(\tau_s+G(\gamma_s))ds+\int_{0}^{t}\zeta_sdB_{s}+K_t,\] where
$K_t=\int_0^t\frac{1}{2}\gamma_sd \langle B\rangle_s-\int_0^tG(\gamma_s)ds$, which, as is known, is a non-increasing $G$-martingale. So a $G$-It\^o process is a generalized $G$-It\^o process.
\end {remark}
By Corollary 3.5 in Song (2012) we
conclude that the decomposition for $G$-It\^o processes is unique. The next result shows the uniqueness of the decomposition for generalized $G$-It\^o processes.

\begin {theorem} \label {gG-Ito}  Assume $\int_0^t\zeta_sdB_s+\int_{0}^{t}\eta_{s}d s+K_{t}= L_{t}$, where
$\zeta\in H^1_G(0,T), \eta \in M^{1}_{G}(0,T)$, and $K_{t}, L_{t}$ are non-increasing
$G$-martingales. Then we have $\int_0^t\zeta_sdB_s\equiv0$, $\int_{0}^{t}%
\eta_{s}d s\equiv0$ and $K_{t}= L_{t}$.
\end {theorem}
\begin {proof} By the uniqueness for the decomposition for (classical) continuous semimartingales, we get $\int_0^t\zeta_sdB_s\equiv0$.
Assume $\int_{0}^{t}\eta_{s}ds+K_{t}=L_{t}$. Since $L_{t}$ is
non-increasing, $\tilde{L}_{t}:=\int_{0}^{t}\eta_{s}^{+}ds+K_{t}$ is
also non-increasing, which implies that $-\int_{s}^{t}\eta_{r}^{+}dr\geq K_t-K_s$ for any $s<t$. Noting that
 $0\geq\mathbb{E}_s[-\int_{s}^{t}\eta_{r}^{+}dr]\geq\mathbb{E}_s[K_{t}-K_{s}]=0$ since $K$ is a $G$-martingale, we conclude that $-\int_{0}^{t}\eta_{s}^{+}ds$ is also a $G$-martingale, which implies, by Theorem \ref {Mthm},
that
$\int_{0}^{t}\eta_{s}^{+}ds=0$. By the same arguments, we have $\int_{0}%
^{t}\eta_{s}^{-}ds=0$.

\end {proof}

\begin {corollary} \label {cor-uni} Assume that $K_t:=\int_0^t\eta_sd \langle B\rangle_s-\int_0^t\zeta_sds$, $\eta, \zeta\in M^1_G(0,T)$, is a non-increasing $G$-martingale. Then we have $\zeta\equiv2G(\eta)$.
\end {corollary}
\begin {proof} Since $L_t:=\int_0^t\eta_sd \langle B\rangle_s-\int_0^t2G(\eta_s)ds=K_t+\int_0^t2G(\eta_s)-\zeta_s ds$ is a non-increasing $G$-martingale, by Theorem \ref {gG-Ito}, we get $\zeta\equiv2G(\eta)$.
\end {proof}

\begin {remark} \label {app} Theorem \ref {gG-Ito} turns out to be a very strong conclusion. Many important results in the context of $G$-expectation theory can be considered as its immediate corollaries.

1) Theorem 3.6 in \cite {Song13}: A process $A_t=\int_0^t\eta_sd \langle B\rangle_s$, $\eta\in M^p_G(0,T)$ for some $p>1$, has stationary and independent increments if and only if $A_t=c \langle B\rangle_t$ for some constant $c\in \mathbb{R}$.

\begin {proof} We only prove the ``only if" part.

Assume that $A$ is a process with stationary and independent increments. Then there exists a constant $\lambda\in \mathbb{R}$ such that $\mathbb{E}[A_t]=\lambda t$ and $L_t:=A_t-\lambda t$ is a non-increasing $G$-martingale. So  we conclude by Corollary \ref {cor-uni} that $\lambda=2G(\eta_s)$, which implies the desired conclusion.
\end {proof}

2) Corollary 3.5 in \cite {Song12}: If $\int_0^t\eta_s d\langle B\rangle_s=\int_0^t\zeta ds$ for some $\eta, \zeta\in M^1_G(0,T)$, we have
$\eta=\zeta \equiv 0.$

\begin {proof} By the assumption, we have
\[\int_0^t\eta_s d\langle B\rangle_s=:K_t+\int_0^t2G(\eta_s)ds=\int_0^t\zeta ds.\]
By Theorem \ref {gG-Ito}, we get $K_t=\int_0^t\eta_sd \langle B\rangle_s-\int_0^t2G(\eta_s)ds\equiv 0.$ For any $\varepsilon\in[0, \frac{\overline{\sigma}^2-\underline{\sigma}^2}{2}]$ we have
 \[0=\mathbb{E}[-K_T]\geq\mathbb{E}_{G_\varepsilon}[-K_T]\geq \varepsilon\mathbb{E}_{G_\varepsilon}[\int_0^T|\eta_s|ds],\] which implies $\eta\equiv0$, and consequently, $\zeta\equiv 0$.
\end {proof}

\end {remark}

\section {Application: characterization of  $G$-Sobolev space $W_{\mathcal{A}_{G}}^{\frac{1}{2},1;p}(0,T)$}

 Peng and Song  (2015) introduced the notion of $G$-Sobolev spaces, in which they defined solutions to a certain type of  path dependent PDEs.

Their definitions of $G$-Sobolev spaces started from the following spaces of smooth functions of paths.
 \begin{definition} \label {cylinderF}
A function $\xi:\Omega_{T}\rightarrow \mathbb{R}$ is called a cylinder function
of paths on $[0,T]$ if it can be represented by
\[
\xi(\omega)=\varphi(\omega(t_{1}),\cdots,\omega(t_{n})),\omega \in \Omega_{T},
\]
for some $0=t_0<\cdots<t_{n}=T$, where $\varphi:(\mathbb{R}^{d}%
)^{n}\rightarrow \mathbb{R}$ is a $C^{\infty}$-function with at most polynomial
growth. We denote by $C^{\infty}(\Omega_{T})$ the collection of all cylinder
functions of paths on $[0,T]$.
\end{definition}
  \begin{definition} A function $u(t,\omega
):[0,T]\times \Omega_{T}\rightarrow \mathbb{R}$ is called a cylinder path
process if there exists a time partition
$0=t_{0}<\cdots<t_{n}=T$, such that for each $k=0,1,\cdots,n-1$ and
$t\in(t_{k},t_{k+1}]$,
\[
u(t,\omega)=u_{k}(t,\omega(t);\omega(t_{1}),\cdots,\omega(t_{k})).
\]
Here for each $k$, the function $u_{k}:[t_{k},t_{k+1}]\times(\mathbb{R}%
^{d})^{k+1}\rightarrow \mathbb{R}$ is a $C^{\infty}$-function with
\[
u_{k}(t_{k},x;x_{1},\cdots,x_{k-1},x)=u_{k-1}(t_{k},x;x_{1},\cdots,x_{k-1})
\]
such that, all derivatives of $u_{k}$ have at most polynomial growth. We
denote by ${\mathcal{C}}^{\infty}(0,T)$ the collection of all cylinder
path processes.
\end{definition} For a function $u\in{\mathcal{C}}^{\infty}(0,T)$, set, for $t\in [t_k, t_{k+1}]$,
\begin{align}\mathcal{D}_{t}u(t,\omega):= & \partial_{t}u_{k}(t,x;x_{1},\cdots,x_{k}%
)|_{x=\omega(t),x_{1}=\omega(t_{1}),\cdots,x_{k}=\omega(t_{k})},\\
\mathcal{D}_{x}u(t,\omega):= &  \partial_{x}u_{k}(t,x;x_{1},\cdots
,x_{k})|_{x=\omega(t),x_{1}=\omega(t_{1}),\cdots,x_{k}=\omega(t_{k}%
)},\label{Dx}\\
\mathcal{D}_{x}^{2}u(t,\omega):= &  \partial_{x}^{2}u_{k}(t,x;x_{1}%
,\cdots,x_{k})|_{x=\omega(t),x_{1}=\omega(t_{1}),\cdots,x_{k}=\omega(t_{k}%
)}.
\end{align}
Set \[\mathcal{A}_{G}u(t,\omega):=\mathcal{D}_{t}%
u(t,\omega)+G(\mathcal{D}_{x}^{2}u(t,\omega)).\]

 \subsection {$G$-Sobolev spaces $W_{G}^{1,2;p}(0,T)$ introduced in \cite {PS15}}
 \begin{definition}
 1) For $u\in{\mathcal{C}}^{\infty}(0,T)$, we set
\[
\Vert u\Vert_{S_{G}^{p}}^{p}=\mathbb{E}[\sup_{s\in \lbrack0,T]}|u(s,\omega)%
|^{p}].
\]
We denote by $S_{G}^{p}(0,T)$ the completion of ${\mathcal{C}}^{\infty
}(0,T)$ w.r.t. the norm $\Vert \cdot \Vert_{S_{G}^{p}}$.

2) For $u\in{\mathcal{C}}^{\infty}(0,T)$, we set
\[
\Vert u\Vert_{W_{G}^{1,2;p}}^{p}=\mathbb{E}[\sup_{s\in \lbrack0,T]}%
|u_{s}|^{p}+\int_{0}^{T}(|\mathcal{D}_{s}u_{s}|^{p}+|\mathcal{D}_{x}^{2}u_{s}|^{p})ds+( \int
_{0}^{T}|\mathcal{D}_{x}u_{s}|^{2}ds)^{p/2}].
\] Denote by $W_{G}^{1,2;p}(0,T)$ the completion of ${\mathcal{C}}^{\infty}(0,T)$
w.r.t. the norm $\Vert \cdot \Vert_{W_{G}^{1,2;p}}$.
\end{definition}
Sometimes, we shall abuse notations by writing $u(t,\omega)$ as $u_t$ for simplicity.

By Corollary 3.5 in Song (2012) we
conclude that the decomposition for $G$-It\^o processes is unique: letting $\tau, \gamma\in M^1_G(0,T), \ \zeta\in H^1_G(0,T)$, then

 \[\int_{0}^{t}\tau_sds+\int_{0}^{t}\zeta_sdB_{s}+\int_0^t\frac{1}{2}\gamma_sd \langle B\rangle_s=0 \] implies that $\tau=\gamma\equiv 0$ and $\zeta\equiv 0$.

From this it is easily seen that the  norm $\Vert \cdot \Vert_{W_{G}^{1,2;p}}$ is closable in the
space $S_{G}^{p}(0,T)$: Let $u^{n}\in{\mathcal{C}}^{\infty}(0,T)$ be a Cauchy
sequence w.r.t. the norm $\Vert \cdot \Vert_{W_{G}^{1,2;p}}$. If $\Vert
u^{n}\Vert_{S_{G}^{p}}\rightarrow0$, we have $\Vert u^{n}\Vert_{W_{G}^{1,2;p}%
}\rightarrow0$.

\begin {remark} \label {keyremark} The closability of the  norm $\Vert \cdot \Vert_{W_{G}^{1,2;p}}$, which follows from the uniqueness of the decomposition for $G$-It\^o processes, is the key point to extend the definition of the operators $\mathcal{D}_{t}$, $\mathcal{D}_{x}$ $\mathcal{D}^2_{x}$ to the space $W_{G}^{1,2;p}(0,T)$. Precisely, unless the  norm $\Vert \cdot \Vert_{W_{G}^{1,2;p}}$ is closable,  a process $u\in S_{G}^{p}(0,T)$ may correspond to two different elements in $W_{G}^{1,2;p}(0,T)$, which can be represented as: $(u, \tau, \zeta, \gamma)$ and $(u, \tilde{\tau}, \tilde{\zeta}, \tilde{\gamma})$, $\tau, \tilde{\tau}, \gamma, \tilde{\gamma} \in M^p_G(0,T)$, $\zeta, \tilde{\zeta}\in H^p_G(0,T)$. For this case, we could not well-define the derivatives for $u$.
\end {remark}

So $W_{G}^{1,2;p}(0,T)$ can be considered as a subspace of $S_{G}^{p}(0,T)$, and the derivative  operators $\mathcal{D}_{t}$, $\mathcal{D}_{x}^2$ (resp.  $\mathcal{D}_{x}$),  can all be extended as continuous linear operators from $W_{G}^{1,2;p}(0,T)$ to $M^p_G(0,T)$ (resp. to $H^p_G(0,T)$).

 \begin {theorem}(Theorem 4.5 in \cite {PS15}) Assume $u\in S_{G}^{p}(0,T)$. Then the following two
conditions are equivalent:

(i) $u\in W_{G}^{1,2;p}(0,T)$;

(ii) There exists $u_0\in {\mathbb{R}}$, $\zeta,w\in M_{G}^{p}(0,T)$ and $v\in H_{G}^{p}(0,T)$ such that
\begin{equation}
u_t=u_0+\int_{0}^{t}\zeta_sds+\int_{0}^{t}%
v_sdB_{s}+\frac{1}{2}\int_{0}^{t}w_sd\langle B\rangle_{s}.\label{uzvw}
\end{equation}
Moreover, we have
\[
\mathcal{D}_{t}u_t=\zeta_t,\,\, \ \mathcal{D}_{x}u_t=v_t,\,\,\ \mathcal{D}_{x}%
^{2}u_t=w_t.
\]
 \end {theorem}

In \cite {PS15}, the authors defined $W_{G}^{1,2;p}(0,T)$-solutions to path dependent PDEs and established one-one correspondence
between backward SDEs under $G$-expectation and a certain type of path dependent PDEs.

\textbf{Backward SDEs:} to find $Y\in S_{G}^{p}(0,T)$, $Z\in H_{G}^{p}(0,T),\eta \in M_{G}%
^{p}(0,T)$ such that
\begin{equation}
Y_{t}=\xi+\int_{t}^{T}f(s,Y_{s},Z_{s},\eta_{s})ds-\int_{t}^{T}Z_{s}%
dB_{s}-(\int_t^T\frac{1}{2}\eta_sd\langle B\rangle_s-\int_t^TG(\eta_s)ds), \label{GBSDE}%
\end{equation}
where  $f:{\normalsize [0,T]\times}\mathbb{R\times R}^{d}%
\times \mathbb{S}{\normalsize (d)\mapsto}\mathbb{R}$ is a given function and
$\xi \in L_{G}^{p}(\Omega_{T})$ is a given random variable.

\textbf{Path Dependent PDEs:} to find a path-dependent $u\in
W_{G}^{1,2;p}(0,T)$ such that
\begin{align}\label{fPPDE}
\begin {split}
\mathcal{D}_{t}u+G(\mathcal{D}_{x}^{2}u)+f(t,u,\mathcal{D}_{x}u,\mathcal{D}_{x}^{2}u)  &  =0,\  \ t\in \lbrack
0,T),\\
u_T  &  =\xi.
\end {split}
\end {align}
We assume that
$f(t,\omega,Y_{t},Z_{t},\eta_{t})\in M_{G}^{p}(0,T)$ for any $(Y,Z,\eta)\in
S_{G}^{p}(0,T)\times H_{G}^{p}(0,T)\times M_{G}^{p}(0,T)$.
\begin{theorem} \label {cor-2nd} (Theorem 4.9 in \cite {PS15})
 Let $(Y,Z,\eta)$ be a solution to the backward SDE
(\ref{GBSDE}). Then we have $u_t:=Y_{t}\in W_{G}^{1,2;p}(0,T)$
with $\mathcal{D}_{x}u_t=Z_{t}$ and $\mathcal{D}_{x}^{2}u_t=\eta
_{t}$.

Moreover, Given a $u \in W_{G}^{1,2;p}(0,T)$, the following conditions
are equivalent:

(i) $(u, \mathcal{D}_{x} u, \mathcal{D}^{2}_{x} u)$ is a solution to the backward SDE (\ref{GBSDE});

(ii) $u$ is a $W_{G}^{1,2;p}$-solution to the path dependent PDE
(\ref{fPPDE}).
\end{theorem}

\subsection {Characterization of  $G$-Sobolev space $W_{\mathcal{A}_{G}}^{\frac{1}{2},1;p}(0,T)$}

Now we consider a special case of the path dependent PDE (\ref{fPPDE}):  \emph{$f$ is independent of $\mathcal{D}^{2}_{x}u$.}
\begin{align}
\begin {split}\label {fPPDE-special}
\mathcal{D}_{t}u+ G(\mathcal{D}^{2}_{x} u)+f(t,u,\mathcal{D}_{x} u)  &  =0, \  \ t\in[0,T),\\
u_T  &  =\xi. %
\end {split}
\end{align}

Let $u\in W_{G}^{1,2;p}(0,T)$ be a solution to the path dependent PDE
(\ref{fPPDE-special}). By Theorem \ref{cor-2nd}, the processes
\[
Y_{t}:=u_t,\,\,\, Z_{t}:=\mathcal{D}_{x}u_t,\,\,\, K_{t}:=\frac{1}{2}\int_{0}^{t}%
\mathcal{D}_{x}^{2}u_sd\langle B\rangle_{s}-\int_{0}^{t}G(\mathcal{D}_{x}^{2}%
u_s)ds
\]
satisfy the following backward SDE:
\begin{equation}
Y_{t}=\xi+\int_{t}^{T}f(s,Y_{s},Z_{s})ds-\int_{t}^{T}Z_{s}dB_{s}-(K_{T}%
-K_{t}), \label{GBSDE-special}%
\end{equation}
which is noting but the  backward SDEs driven by $G$-Brownian motion ($G$-BSDE)  studied in \cite{HJPS}.

On the contrary, letting $(Y,Z,K)$ be a solution of backward SDE (\ref{GBSDE-special}) considered in \cite{HJPS}, notice that, although we have many interesting examples, but it is still a very  interesting and challenging problem to give  reasonable conditions on $\xi$ and $f$ such that  $Y$ lies in the Sobolev space $W_{G}^{1,2;p}(0,T)$. Even so, we still think $u=Y$ is a reasonable candidate of the solution to Equ. (\ref {fPPDE-special}).

 In \cite {PS15}, the authors  formulated $u=Y$  as the unique
solution to Equ. (\ref {fPPDE-special}) in a first order Sobolev space $W_{\mathcal{A}_{G}}^{\frac{1}%
{2},1;p}(0,T)$. In this section, we shall refine the definition of the $G$-Sobolev space $W_{\mathcal{A}_{G}}^{\frac{1}%
{2},1;p}(0,T)$. The main idea is, just like the liner case, to integrate $\mathcal{A}_Gu=\mathcal{D}_{t}u+ G(\mathcal{D}^{2}_{x} u)$ as one operator, which reduces the regularity requirement for the solutions. To well define the derivative $\mathcal{A}_Gu$ for $u\in W_{\mathcal{A}_{G}}^{\frac{1}%
{2},1;p}(0,T)$, the uniqueness of the decomposition for generalized $G$-It\^o processes plays a crucial role.

\subsubsection {Definition of the $G$-Sobolev space $W_{\mathcal{A}_{G}}^{\frac{1}%
{2},1;p}(0,T)$ }

 For
$u,v\in {\mathcal{C}}^{\infty}(0,T)$, set
\[
d_{G,p}^{p}(u,v)=\mathbb{E}[\sup_{s\in\lbrack0,T]}|u_{s}-v_{s}%
|^{p}+(\int_{0}^{T}|\mathcal{D}_{x}(u_{s}-v_{s})|^{2}ds)^{\frac{p}{2}}+\int_{0}^{T}|{\mathcal{A}}_{G}u_{s}-{\mathcal{A}}_{G}v_{s}|^{p}ds].
\]

By the uniqueness of the  decomposition for generalized $G$-It\^o processes we obtain the closability of the metric $d_{G,p}$. As is stated in Remark \ref {keyremark}, closability of the metric $d_{G,p}$ is the key point to well define the operators $\mathcal{A}_G$, $\mathcal{D}_x$.

\begin{proposition}
 The metric $d_{G,p}$ is closable in the space $S_{G}%
^{p}(0,T)$: Let $u^{n}, v^{n}\in{\mathcal{C}}^{\infty}(0,T)$ be two
Cauchy
sequences w.r.t. the metric $d_{G,p}$. If $\|u^{n}-v^{n}\|_{S_{G}^{p}%
}\rightarrow0$, we have $d_{G,p}(u^{n},v^{n})\rightarrow0$.
\end{proposition}
\begin {proof} For $u^n\in {\mathcal{C}}^{\infty}(0,T)$, by It\^o's formula, we have
\begin{align*}
u^n(t,\omega)&  =u^n(0,\omega)+\int_{0}^{t}\mathcal{D}_{s}u^n(s,\omega)ds+\int_{0}^{t}%
\mathcal{D}_{x}u^n(s,\omega)dB_{s}+\frac{1}{2}\int_{0}^{t}\mathcal{D}_{x}^{2}u^n(s,\omega)d\langle
B\rangle_{s}\\
&  =u^n(0,\omega)+\int_{0}^{t}\mathcal{A}_{G}u^n(s,\omega)ds+\int_{0}^{t}%
\mathcal{D}_{x}u^n(s,\omega)dB_{s}+K^n_{t},
\end{align*}
where
 $K^n_{t}:=\frac{1}{2}\int_{0}^{t}\mathcal{D}_{x}^{2}u^n(s,\omega)d\langle B\rangle_{s}%
-\int_{0}^{t}G(\mathcal{D}_{x}^{2}u^n(s,\omega))ds$ is a non-increasing $G$-martingale.
If $(u^n)_{n}$ is  a Cauchy
sequence w.r.t. the metric $d_{G,p}$, there will be processes $u\in S^p_G(0,T)$, $\eta\in M^p_G(0,T)$, $\zeta\in H^p_G(0,T)$ such that
\[\|u^n-u\|_{S^p_G}+\|\mathcal{A}_{G}u^n-\eta\|_{M^p_G}+\|\mathcal{D}_{x}u^n-\zeta\|_{H^p_G}\rightarrow0.\]
Set $K_t=u_t-u_0-\int_0^t\eta_sds-\int_0^t\zeta_s dB_s$. It is easily seen that $\|K^n-K\|_{S^p_G}\rightarrow 0$. So $K$ is a non-increasing $G$-martingale and $u_t=u_0+\int_0^t\eta_sds+\int_0^t\zeta_s dB_s+K_t$ is a generalized $G$-It\^o process. Assuming $(v_n)_n$ is a Cauchy
sequence w.r.t. the metric $d_{G,p}$, similarly, there exists a generalized $G$-It\^o process $\tilde{u}_t=\tilde{u}_0+\int_0^t\tilde{\eta}_sds+\int_0^t\tilde{\zeta}_s dB_s+\tilde{K}_t$ such that \[\|v^n-\tilde{u}\|_{S^p_G}+\|\mathcal{A}_{G}v^n-\tilde{\eta}\|_{M^p_G}+\|\mathcal{D}_{x}v^n-\tilde{\zeta}\|_{H^p_G}\rightarrow0.\]
If $\|u^{n}-v^{n}\|_{S_{G}^{p}%
}\rightarrow0$, we get $u=\tilde{u}$. By the uniqueness of the decomposition for generalized $G$-It\^o processes, we get $\eta=\tilde{\eta}$ and $\zeta=\tilde{\zeta}$, which implies $d_{G,p}(u^{n},v^{n})\rightarrow0$.

\end {proof}

Denote by $W_{\mathcal{A}_{G}}^{\frac{1}%
{2},1;p}(0,T)$ the closure of
${\mathcal{C}}^{\infty}(0,T)$ w.r.t. the metric $d_{G,p}$ in
$S_{G}^{p}(0,T)$. Now the operators ${\mathcal{A}}_{G}, \mathcal{D}_{x}$ can
be continuously extended to the space $W_{\mathcal{A}_{G}}^{\frac{1}%
{2},1;p}(0,T)$.

\begin{proposition}
\label{Sobolev-1st}Assume $u\in S_{G}^{p}(0,T)$. Then the following two
conditions are equivalent:

(i) $u\in W_{\mathcal{A}_{G}}^{\frac{1}%
{2},1;p}(0,T)$;

(ii) There exists  $\eta\in M^p_G(0,T)$ and $\zeta\in H^p_G(0,T)$  such that

\[ u(t, \omega)-\int_{0}^{t}\eta(s, \omega)ds-\int_{0}%
^{t}\zeta(s,\omega)dB_s \] is  a non-increasing $G$-martingale, namely, $u$ is a generalized  $G$-It\^o process.

Moreover, we have ${\mathcal{A}}_{G}u=\eta$ and $\mathcal{D}_{x} u=\zeta$.
\end{proposition}

\begin{proof}
(i)$\implies$(ii) is obvious.

(ii)$\implies$(i). Let $u$ be a generalized $G$-It\^o process. By Theorem 5.4 in \cite {Song11b}, it suffices to prove the claim for $u$ of the following
form:
\[
u_t=u_0+\int_{0}^{t}\eta_{s}ds+\int_{0}^{t}\zeta_{s}%
dB_{s}+\frac{1}{2}\int_{0}^{t}w_{s}d\langle B\rangle_{s}-\int_{0}^{t}%
G(w_{s})ds,
\]
 where $\eta,\zeta,w$ are smooth step processes, namely, $\eta, \zeta, w\in M^0(0,T)$ and $\eta_t, \zeta_t, w_t\in C^\infty(\Omega_t)$ (Def. \ref{cylinderF}).  Set $t_{k}^{n}=\frac{kT}{2^{n}}$ and
$
Q^{n}_t:=\sum_{k=0}^{2^{n}-1}(B_{t_{k+1}^{n}\wedge t}-B_{t_{k}%
^{n}\wedge t})^{2}=\int_{0}^{t}\lambda^{n}_sdB_{s}+\langle
B\rangle_{t},
$
where $\lambda^{n}_t=\sum_{k=0}^{2^{n}-1}2(B_{t}-B_{t_{k}}%
)1_{]t_{k},t_{k+1}]}(t)$. Choose a sequence of smooth step processes $\alpha_{s}^{n}$
such that
$\mathbb{E}[\int_{0}^{T}|\alpha_{s}^{n}-G(w_{s})|^{p}ds]\rightarrow0.$
Set
\begin{eqnarray}
\label {gGIto-e1}u^n_t:=& &
u_{0}+\int_{0}^{t}\eta_{s}ds+\int_{0}^{t}\zeta_{s}dB_{s}+\frac
{1}{2}\int_{0}^{t}w_{s}dQ_{s}^{n}-\int_{0}^{t}\alpha_{s}^{n}ds\\
\label {gGIto-e2}=& &u_{0}+\int_{0}^{t}(\eta_{s}-\alpha_{s}^{n})ds+\int_{0}^{t}(\zeta_{s}+\frac
{1}{2}w_{s}\lambda_{s}^{n})dB_{s}+\int_{0}^{t}\frac{1}{2}w_{s}d\langle
B\rangle_{s}.
\end{eqnarray}
It is easily seen, by (\ref{gGIto-e1}), that
$u^{n}$ belongs to ${\mathcal{C}}^{\infty}(0,T)$. By the uniqueness of the decomposition for $G$-It\^{o} processes and (\ref{gGIto-e2}) we know that
\begin{align*}
\mathcal{D}_{t}u^{n}_t=\eta_{t}-\alpha_{t}^{n}, \
\mathcal{D}_{x}u^{n}_t=\zeta_{t}+\frac{1}{2}w_{t}\lambda_{t}^{n}, \
\mathcal{D}_{x}^{2}u^{n}_t=w_{t}.
\end{align*}
So
${\mathcal{A}}_{G}u^{n}_t=\eta_{t}-\alpha_{t}^{n}+G(w_{t})$.
It is easy to show that
$\mathcal{A}_{G}u^{n}\stackrel {M^p_G}{\longrightarrow}\eta, \ D_{x}u^{n}\stackrel {H^p_G}{\longrightarrow} \zeta, \ u^{n}\stackrel {S^p_G}{\longrightarrow} u.$
So $u\in W_{\mathcal{A}_G}^{\frac{1}{2},1;p}(0,T)$ with ${\mathcal{A}}_{G}u=\eta$ and
$\mathcal{D}_{x}u=\zeta$.
\end{proof}

\subsubsection{Fully nonlinear path dependent PDEs}

Let us define the $W^{\frac{1}{2},1,p}_{\mathcal{A}_G}$-solution to the path dependent PDE
(\ref{fPPDE-special}) : to find $u\in W_{\mathcal{A}_{G}}^{\frac{1}%
{2},1;p}(0,T)$ such that
\begin{align}
\begin {split}\label {fPPDE-special-weak}
{\mathcal{A}}_{G}u+f(t,u,\mathcal{D}_{x}u)
&
=0,\ \ t\in\lbrack0,T),\\
u_T  &  =\xi.
\end {split}
\end{align}
Now we can interpret backward SDEs driven by $G$-Brownian motion as ``path dependent" PDEs.

We assume that
$f(t,\omega,Y_{t},Z_{t})\in M_{G}^{p}(0,T)$ for any $(Y,Z)\in
S_{G}^{p}(0,T)\times H_{G}^{p}(0,T)$.
\begin{theorem}
\label{cor-1st} Let  $(Y,Z)$ be  a solution to the backward
SDE (\ref{GBSDE-special}) (see Def. \ref {defA1} and Rem. \ref {remark-A1}). Then we have $u_t:=Y_{t}\in W_{\mathcal{A}_{G}}^{\frac{1}%
{2},1;p}(0,T)$ with
$\mathcal{D}_{x}u_t=Z_{t}$.

Moreover, Given a $u \in W_{\mathcal{A}_{G}}^{\frac{1}%
{2},1;p}(0,T)$, the following conditions
are equivalent:

(i) $(u, \mathcal{D}_{x} u)$ is a solution to the backward SDE (\ref{GBSDE-special});

(ii) $u$ is a $W_{\mathcal{A}_{G}}^{\frac{1}%
{2},1;p}(0,T)$-solution to the path dependent PDE
(\ref{fPPDE-special-weak}).
\end{theorem}
\begin {proof} The proof follows immediately from Proposition \ref {Sobolev-1st} and the definitions of the solutions to the backward SDE (\ref{GBSDE-special}) and the path dependent PDE
(\ref{fPPDE-special-weak}).
\end {proof}

Assume that the function $g(t,\omega ,y,z):[0,T]\times \Omega_{T}\times R\times
R\rightarrow R$ satisfies the following assumptions: there exists some
$\beta>1$ such that

\begin{description}
\item[(H1)] for any $y$, $z$, $g(t,\omega ,y,z)\in M_{G}^{\beta
}(0,T)$;

\item[(H2)] $|g(t,\omega ,y,z)-g(t,\omega ,y^{\prime},z^{\prime
})|\leq L(|y-y^{\prime}|+|z-z^{\prime}|)$ for some constant $L>0$.
\end{description}

\begin{corollary}
 Assume $\xi \in L_{G}^{\beta}(\Omega_{T})$ and $g$ satisfies
(H1) and (H2) for some $\beta>1$. Then, for each
$p\in(1,\beta)$, the path dependent PDE (\ref{fPPDE-special}) has a
unique $W_{\mathcal{A}_{G}}^{\frac{1}{2},1;p}$-solution $u$.

In particular, the martingale $u(t,\omega):={\mathbb{E}}_t[\xi](\omega)$ is the unique $W_{\mathcal{A}_{G}}^{\frac{1}{2},1;p}$-solution of the path dependent $G$-heat equation
\[
\mathcal{D}_tu+G(\mathcal{D}_x^2u)
=0,\,\,\,\, u_T=\xi.
\]
\end{corollary}

\begin{proof}
The uniqueness is straightforward from Theorem \ref{cor-1st} and Theorem \ref {thmA}.

We now prove the existence. By Theorem \ref {thmA} we know that the backward SDE (\ref{GBSDE-special}%
) has a solution $(Y,Z)$. By the assumption (H1) and
(H2), we conclude $g(t,\omega ,Y_{t}(\omega),Z_{t}(\omega))\in
M_{G}^{p}(0,T)$. So we get the existence result from Theorem \ref{cor-1st}.

By the $G$-martingale decomposition theorem, $u\in S_{G}^{p}(0,T)$ is a
$G$-martingale if and only if $u$ is a solution of backward SDE
(\ref{GBSDE-special}) with $f=0$.  So $u(t,\omega):={\mathbb{E}}_t[\xi](\omega)$ is the unique $W_{\mathcal{A}_{G}}^{\frac{1}{2},1;p}$-solution of the path dependent $G$-heat equation.
\end{proof}

\section{Appendix: Backward SDEs driven by $G$-Brownian motion}

In \cite{HJPS} the authors studied the backward stochastic differential
equations driven by a $G$-Brownian motion $(B_{t})_{t\geq0}$ in the following
form:
\begin{equation}
Y_{t}=\xi+\int_{t}^{T}g(s,Y_{s},Z_{s})ds-\int_{t}^{T}Z_{s}dB_{s}-(K_{T}%
-K_{t}). \label{equA}%
\end{equation}
where $K$ is a non-increasing $G$-martingale.

The main result in \cite {HJPS} is the existence and uniqueness of a solution
$(Y,Z, K)$ for equation (\ref{equA}) in the $G$-framework under the following
assumption: there exists some $\beta>1$ such that (H1) and (H2) are satisfied.

\begin{definition}
\label{defA1} Let $\xi \in L_{G}^{\beta}(\Omega_{T})$ and $g$ satisfy (H1) and
(H2) for some $\beta>1$. A triplet of processes $(Y,Z,K)$ is called a solution
of equation (\ref{equA}) if for some $1<\alpha \leq \beta$ the following
properties hold:

\begin{description}
\item[(a)] $Y\in S_{G}^{\alpha}(0,T)$, $Z\in H_{G}^{\alpha}(0,T)$, $K$ is a
non-increasing $G$-martingale with $K_{0}=0$ and $K_{T}\in L_{G}^{\alpha
}(\Omega_{T})$;

\item[(b)] $Y_{t}=\xi+\int_{t}^{T}g(s,Y_{s},Z_{s})ds-\int_{t}^{T}Z_{s}%
dB_{s}-(K_{T}-K_{t})$.
\end{description}
\end{definition}

The main result in \cite {HJPS} is the following theorem:

\begin{theorem}
\label{thmA} Assume that $\xi \in L_{G}^{\beta}(\Omega_{T})$ and $f$ satisfies
(H1) and (H2) for some $\beta>1$. Then equation (\ref{equA}) has a unique
solution $(Y,Z,K)$. Moreover, for any $1<\alpha<\beta$ we have $Y\in
S_{G}^{\alpha}(0,T)$, $Z\in H_{G}^{\alpha}(0,T)$ and $K_{T}\in L_{G}^{\alpha
}(\Omega_{T})$.
\end{theorem}
\begin {remark} \label {remark-A1} Equivalently, we say a pair of processes $(Y,Z)$ is  a solution
of equation (\ref{equA}) if for some $1<\alpha \leq \beta$ the following
properties hold:

\begin{description}
\item[(a)] $Y\in S_{G}^{\alpha}(0,T)$, $Z\in H_{G}^{\alpha}(0,T)$;

\item[(b)] $Y_T=\xi$ and $K_t:=Y_t+\int_{0}^{t}g(s,Y_{s},Z_{s})ds-\int_{0}^{t}Z_{s}%
dB_{s}$  is a
non-increasing $G$-martingale.
\end{description}
\end {remark}


\renewcommand{\refname}{\large References}{\normalsize \ }

\end{document}